\documentclass[10pt,english]{amsart}

\usepackage{amssymb}
\usepackage{amsfonts,latexsym}
\usepackage{color}

\newtheorem{theorem}   {Theorem}[section]
\newtheorem{proposition}  {Proposition}[section]

\newtheorem{lemma} {Lemma}[section]
\newtheorem{remark}   {Remark}[section]
\newtheorem{definition}  {Definition}[section]

\begin{document}

\title[Solutions of  elliptic equations in tubes ]{Solutions of semilinear elliptic
equations in tubes}

\author[F.\ Pacard]{Frank Pacard $^+$}
\address{
Centre de Math\'ematiques Laurent Schwartz, \'Ecole Polytechnique, 91128 Palaiseau, France}
\email{frank.pacard@math.polytechnique.fr}

\author[F.\ Pacella]{Filomena Pacella $^*$}
\address{
Dipartimento di Matematica,  Universit\`{a} "Sapienza" di Roma, P.le. Aldo Moro 2, 00184 Rome, Italy}
\email{pacella@mat.uniroma1.it}

\author[B.\ Sciunzi]{Berardino Sciunzi$^\diamond$}
\address{
Dipartimento di Matematica,
Universit\`a della Calabria,
Ponte Pietro Bucci 31B, 87036 Cosenza, Italy}
\email{sciunzi@mat.unical.it}

\thanks{The first author is partially supported by the ANR-08-BLANC-0335-01\,grant. The second author is partially supported by PRIN-2009-WRJ3W7\,grant.
The third author is partially supported by PRIN-2009\,grant: {\em Metodi Variazionali e Topologici
nello Studio di Fenomeni non Lineari}, and ERC-2011-grant: \emph{Elliptic PDE's and symmetry of interfaces and layers for odd nonlinearities.}}

\maketitle

\begin{abstract}
Given a smooth compact $k$-dimensional manifold $\Lambda$ embedded in $\mathbb {R}^m$, with $m\geq 2$ and $1\leq k\leq m-1$, and given $\epsilon>0$, we define $B_\epsilon (\Lambda)$ to be the geodesic tubular neighborhood of radius $\epsilon$ about $\Lambda$. In this paper, we construct positive solutions of  the semilinear elliptic equation
\begin{equation}\nonumber
\left\{
\begin{array}{rlllll}
\Delta u  + u^p &=& 0     \qquad \quad {\rm in} \quad B_\epsilon (\Lambda) \\[3mm]
 u & = & 0 \quad\qquad {\rm on} \quad \partial B_\epsilon (\Lambda)  ,
\end{array}
\right.
\end{equation}
when the parameter $\epsilon$ is chosen small enough.  In this equation, the exponent $p$ satisfies either $p > 1$ when $n:=m-k \leq 2$ or $p\in (1, \frac{n+2}{n-2})$ when $n>2$. In particular $p$ can be critical or supercritical in dimension $m\geq 3$. As $\epsilon$ tends to $0$, the solutions we construct have Morse index tending to infinity. Moreover, using a Pohozaev type argument, we prove that our result is sharp in the sense that there are no positive solutions for $p>\frac{n+2}{n-2}$, $n\geq 3$, if $\epsilon$ is sufficiently small.
\end{abstract}

\section{Introduction}

Assume that we are given $\Lambda$, a smooth compact $k$-dimensional submanifold which is  embedded in $\mathbb R^m$, where $m\geq 2$ and $k\in \{1, \ldots , m-1\}$. For all $\epsilon>0$, we define the plain tubular neighborhood of radius $\epsilon$, centered about $\Lambda$ by
\begin{equation}
\label{problemdariso}
B_\epsilon  (\Lambda)  : = \left\{  x \in \mathbb R^m  \, : \, {\rm dist} (x, \Lambda) < \epsilon \right\},
\end{equation}
where $x \mapsto {\rm dist} (x,Y)$ denotes the Euclidean distance in $\mathbb{R}^m$ from $x$ to $Y$. Let us observe that, for all $\epsilon$ small enough, the boundary  of $B_\epsilon (\Lambda)$, which is defined by
\[
T_\epsilon (\Lambda) : = \{ x \in \mathbb R^m  \, :  \,  {\rm dist} (x, \Lambda) = \epsilon \},
\]
is a smooth embedded hypersurface in $\mathbb{R}^m$.

The aim of the paper is to show, for $\epsilon$ small, the existence of a new family of positive solutions of the semilinear elliptic problem
\begin{equation}
\label{mainproblemepsilon}
\left\{
\begin{array}{rlllll}
\Delta u  + u^p &=& 0  \qquad \quad {\rm in} \quad B_\epsilon (\Lambda)  \\[3mm]
u & = & 0 \quad\qquad {\rm on} \quad T_\epsilon (\Lambda)  .
\end{array}
\right.
\end{equation}
Note that, in this paper $\Delta$ will  always represent the Laplace operator in $\mathbb{R}^n$.

To state precisely our result we need some preliminaries. Let
\[
n  : = m-k \geq 1,
\]
denote the codimension of $\Lambda$ in $\mathbb{R}^m$. It is well known that, if $B_1^n$ denotes the $n$-dimensional unit ball in $\mathbb{R}^n$, there exists a positive solution $U$ of
\begin{equation}\label{problemradialguida}
\left\{
\begin{array}{rlllll}
\Delta U+ U^p & = & 0  \qquad {\rm in} \quad B_1^n \\[3mm]
 	   U & = & 0 \qquad {\rm on} \quad \partial B_1^n ,
\end{array}
\right.
\end{equation}
provided the exponent $p$ is chosen to satisfy $1<p<+\infty$ when $n\leq 2$ or $p\in(1,\frac{n+2}{n-2})$ when $n\geq 3$. Thanks to Gidas-Ni-Nirenberg's Theorem \cite{GNN}, the function $U$ is  known to be radially symmetric. Moreover, this solution is known to be unique, nondegenerate (we refer to Theorem 4.1 and Theorem 4.2 in \cite{DGP} for a proof of this fact) and to have Morse index equal to $1$.

For all $\epsilon >0$, we define
\begin{equation}\label{defspread1}
\bar u_\epsilon(x) :=  \epsilon^{-\frac{2}{p-1}} U \left( \frac{\text{dist} (x, \Lambda)}{\epsilon} \right)  ,
\end{equation}
for all $x \in B_\epsilon (\Lambda)$. This function is obtained by translating a rescaled copy of $U$ along the manifold $\Lambda$. With these notations at hand, we have the following~:

\begin{theorem}
\label{maintheoremenunciato}
Assume that $$
p\in(1, + \infty), \quad  \mbox{if} \quad  n\leq 2 ,\quad \mbox{or} \quad  p\in \left( 1, \frac{n+2}{n-2} \right), \quad \mbox{if} \quad n\geq 3,
$$
where $n:=m-k$. Then, there exists $\bar \epsilon>0$ and  $\mathcal S \subset (0, \bar\epsilon)$ such that~:

\begin{enumerate}
\item For all $\epsilon \in \mathcal S$, there exists a positive solution $u_\epsilon$ of \eqref{mainproblemepsilon} satisfiying
\begin{equation}
\lim_{\epsilon \to 0, \,  \epsilon \in \mathcal S} \left\| \frac{u_\epsilon}{\bar u_\epsilon} \right\|_{L^\infty (B_\epsilon(\Lambda))}  = 1.
\label{eq:pa0}
\end{equation}

\item For all $\alpha \geq 1$,
\begin{equation}
\lim_{\epsilon \to 0}\frac{1}{\epsilon^\alpha} \left( \epsilon -  \mbox{\rm meas } \mathcal S_\epsilon \right) =0,
\label{pa:1}
\end{equation}
where $\mathcal S_\epsilon : = \mathcal S \cap (0,\epsilon)$.
\item As $\epsilon \in \mathcal S$ tends to $0$, the Morse index of $u_\epsilon$ tends to infinity.
\end{enumerate}
\end{theorem}

Let us briefly comment on our result and in particular on the structure of the set $\mathcal S$ in which the parameter $\epsilon$ can be chosen. As will be apparent in the proof, our construction does not hold for all values of the parameter $\epsilon$ close to $0$. There is a resonance phenomenon which prevents the construction to hold for any small value of $\epsilon$ and which forces $\epsilon$ to be taken away from a set of small density close to $0$. This is precisely the meaning of (\ref{pa:1}). Such a phenomenon  is not new and, in the context of semilinear partial differential equations, it was originally found by A. Malchiodi and M. Montenegro in \cite{Mal-Mon-1}. Since this seminal paper, this phenomenon has also been found in other instances, for example in the study of other semilinear partial differential equations \cite{Mah-Mal, Mal-Mon-2} or in the study of constant mean curvature surfaces \cite{Mah-Maz-Pac, Maz-Pac}.

The Morse index of $u_\epsilon$ is defined to be equal to the dimension of the subspace of $H^1_0(B_\epsilon (\Lambda))$ over which the quadratic form
\[
v \mapsto \int_{B_\epsilon (\Lambda)} \left( |\nabla v|^2 - p   u_\epsilon^{p-1}   v^2\right) dx ,
\]
is definite negative. The fact that we are not able to construct the solutions for all values of $\epsilon$ close enough to $0$ is also reflected in another important feature of our solutions, namely that their Morse index tends to infinity as $\epsilon$ tends to $0$.

In the same way, recall that when $p=\frac{m+2}{m-2}$ and $k=m-1$ it has been proved in \cite{AMPAC} that the energy and the Morse index of all positive solutions tend to infinity as $\epsilon$ tends to $0$.

The shape of the solution we construct is also worth mentioning, in fact, the solution $u_\epsilon$ is close to the function $\bar u_\epsilon$ which has been defined in \eqref{defspread1} and hence it does not concentrate at points as $\epsilon\rightarrow 0$. Also,  by \eqref{defspread1}, we have
\[
\|\bar u_\epsilon\|_{L^\infty( B_\epsilon (\Lambda))} = \mathcal O (\epsilon^{-\frac{2}{p-1}}),
\]
as $\epsilon$ tends to $0$.

In the particular case where the exponent $p$ is the critical Sobolev exponent i.e. when $p=\frac{m+2}{m-2}$, a well known theorem by A. Bahri and J.M. Coron \cite{BC} yields the existence of positive solutions of \eqref{mainproblemepsilon}, provided that the topology of the domain is not trivial. In this case, our result can be seen as a direct construction of a positive solution, via a technique which also gives the shape of the solution, that cannot be deduced from the proof in \cite{BC}.

The solutions we construct are also new in the subcritical case ($p<\frac{m+2}{m-2}$) since they are qualitatively different from the so called multibump solutions which were found in \cite{ACP,DY} and which do not satisfy (\ref{eq:pa0}).

Let us observe that the result of Theorem \ref{maintheoremenunciato} holds for  supercritical exponents, namely exponents which are larger than the critical Sobolev exponent  $\frac{m+2}{m-2}$ in dimension $m$. In the particular case where the codimension $n$ of the manifold $\Lambda$ is equal to $1$ or $2$, the exponent $p$ can be taken to be arbitrarily large. To our knowledge this is the first existence result for solutions of \eqref{mainproblemepsilon} defined in tubular neighborhoods of general $k$-dimensional manifolds. A previous result has been recently obtained for $(m-1)$-dimensional manifolds in \cite{BCGP} and it was indeed a source of inspiration for the present paper.

As it will become clear in the proof, the smoothness of the submanifold $\Lambda$ is a key ingredient of the proof. However, close inspection also shows that this assumption can be relaxed if one is ready to loose some control on the density of the set $\mathcal S$ close to $0$. Indeed, we have the~:
\begin{proposition}
Under the assumptions of Theorem~\ref{maintheoremenunciato}, if in point (2) of Theorem~\ref{maintheoremenunciato} we fix $\alpha \geq 1$, then there exists $l \in \mathbb N$ only depending on $n$ and $\alpha$, such that the conclusion of Theorem~\ref{maintheoremenunciato} holds provided $\Lambda$ is at least a $\mathcal C^l$ submanifold. The larger $\alpha$ is, the larger $l$ has to be chosen.
\end{proposition}

Let us emphasize on the fact that the scheme of the proof is not new and in fact it is inspired from \cite{Mal-Mon-1} and \cite{Mah-Maz-Pac}. However, our framework is simpler and we hope that this will help the interested reader to understand the ideas and techniques in these more involved works.\\

When $n\geq 3$, the existence result of Theorem~\ref{maintheoremenunciato} holds under the assumption that
\[
p < \frac{n+2}{n-2}.
\]
Note that $\frac{n+2}{n-2}$ is the critical Sobolev exponent  in dimension $n$ and observe that this assumption is used to construct the approximate solution $\bar u_\epsilon$ to \eqref{mainproblemepsilon}. One might wonder whether if this condition is only technical. As we will see this is not the case and the existence of positive solutions to \eqref{mainproblemepsilon} generally fails if $p>\frac{n+2}{n-2}$ as we prove by a Pohozaev type argument.

\begin{theorem}\label{nonexistence}
Assume that $n = m- k \geq 3$ and that
\[
p>\frac{n+2}{n-2}.
\]
Then, there exists $\bar \epsilon=\bar \epsilon(p)>0$, such that for all $\epsilon\in (0, \bar\epsilon)$ there is no bounded positive solution of \eqref{mainproblemepsilon} in $B_\epsilon (\Lambda)$.
\end{theorem}
The proof of Theorem\ref{nonexistence} relies on a Pohozaev type identity which we derive for solutions of  \eqref{mainproblemepsilon}. This is a standard techniques which has been used in several nonexistence results and it goes back to \cite{Poh} where the case of star-shaped domains was considered. A similar idea was already used by D. Passaseo in \cite{Pas2}, with a more involved construction, leading to nonexistence results for superlinear elliptic problems in topologically nontrivial domains. In our case, the use of suitable coordinates, namely Fermi coordinates (see Section 3), proves to be extremely useful to get Theorem \ref{nonexistence} in a simple way. We emphasize that as in  \cite{Pas2}, our domains are not star-shaped and are not topologically trivial.

\section{Outline of the proof of Theorem~\ref{maintheoremenunciato}}

The proof of Theorem \ref{maintheoremenunciato} consists in showing that there exists a \emph{genuine solution} $u_\epsilon$ near the \emph{approximate solution} $\bar u_\epsilon$ defined in \eqref{defspread1} provided the parameter $\epsilon$ is chosen small enough and away from a set where resonance occurs. The main steps of the proof are the following~:

\begin{itemize}
\item [(i)]
First, we observe that
\[
\| \Delta \bar u_{\epsilon} +  \bar u_{\epsilon}^p \|_{L^\infty(B_\epsilon (\Lambda))} \leq C  \epsilon^{- \frac{p+1}{p-1}}.
\]
Then, using a finite step iteration scheme, we  improve the approximate solution $\bar u_\epsilon$ into a sequence of approximate solutions $(u_{\epsilon,i})_{i\in\mathbb N}$, which are as close as we want from a genuine solution of the equation in the sense that
\[
\| \Delta u_{\epsilon, i} + u_{\epsilon, i}^p \|_{L^\infty(B_\epsilon (\Lambda))} \leq C   \epsilon^{i - \frac{p+1}{p-1}}.
\]
Moreover, the sequence  $(u_{\epsilon,i})_{i\in \mathbb N}$ is constructed in such a way  that one has a good control on the difference  $u_{\epsilon,i} - \bar u_\epsilon$. As already mentioned,  the construction of $u_{\epsilon,i}$ relies on some iteration scheme and we will see that, in order to keep a good control on the sequence of approximate solutions, we need to allow a loss of regularity at each iteration. In particular, the fact that the sequence $(u_{\epsilon, i})_{i \in \mathbb N}$ exists for all $i \in \mathbb N$, uses the fact that the manifold $\Lambda$ is smooth in an essential way. If the submanifold $\Lambda$ has only finite regularity, the  sequence can just be constructed for a finite number of indices.

\item[(ii)]
Next, we study the linearized operator
\[
L_{\epsilon, i} : = \Delta + p u_{\epsilon, i}^{p-1},
\]
about the approximate solution $u_{\epsilon,i}$ and show that the norm of the inverse of $L_{\epsilon, i}$ can be controlled as $\epsilon$ tends to $0$, provided $\epsilon$ is taken away from a countable sequence tending to $0$. More precisely, we will see that the Morse index of $L_{\epsilon, i}$ tends to infinity as $\epsilon$ tends to $0$. In particular, for fixed $i\in \mathbb N$, the operator $L_{\epsilon, i}$ is not  invertible for a sequence of values $\epsilon$ tending to $0$, and, in order to proceed, we will need to take $\epsilon$ away from these values.

\item[(iii)]  Finally, we look for a genuine solution of \ref{mainproblemepsilon} in the form
\[
u_\epsilon=u_{\epsilon,i} + \varphi _\epsilon .
\]
At this stage, we show that we can rephrase the problem as a fixed point problem which can easily be solved using the fixed point theorem for contraction mappings.
\end{itemize}

The outline of the paper is the following. In section 3 we describe the notations we use and state  basic results about the structure of the Laplace operator in Fermi coordinates about $\Lambda$. In section 4, we construct the sequence of approximate solutions $(u_{\epsilon,i})_{i \in \mathbb N}$ and derive the relevant estimates. In Section 5 and 6, we analyze the spectrum and the uniform invertibility of $L_{\epsilon, i}$, the linearized operator about  $u_{\epsilon,i}$. In section 7, we complete the proof of Theorem \ref{maintheoremenunciato} by  reducing the problem to the solvability of a fixed point problem for contraction mappings. In section 8  we prove Theorem \ref{nonexistence}.

\section{Fermi coordinates near $\Lambda$}\label{preliminaries}

An important tool in the proof of Theorem~\ref{maintheoremenunciato} is the use of appropriate coordinates to parameterize $B_\epsilon (\Lambda)$. We identify $\Lambda$ with the zero section of $N\Lambda$ the normal bundle of $\Lambda$ and $B_\epsilon (\Lambda)$ will be identified with
\[
\Omega_\epsilon (\Lambda) : = \{ (y, z) \in  N\Lambda \,  : \,  y \in \Lambda, \quad z \in N_y \Lambda , \quad  | z | < \epsilon\},
\]
via the mapping
\[
\begin{array}{rccclll}
F_1  : & \Omega_\epsilon(\Lambda) & \to & B_\epsilon (\Lambda)\\[3mm]
& (y,z) & \mapsto & y + z .
\end{array}
\]
The normal bundle $N\Lambda$ is endowed with the metric induced by the embedding of $\Lambda$ in $\mathbb R^m$, namely
\[
\bar g = \mathring g  + g_z ,
\]
where  $\mathring g$ is the induced metric on $\Lambda$ and $g_z :=  dz^2$ the (Euclidean) metric on the normal fibers.

In a neighborhood of a given point $y \in \Lambda$, we can define a moving orthonormal frame
\[
e^1, \ldots , e^n \in N\Lambda  ,
\]
where each $e^j$ is a smooth section of the normal bundle $N\Lambda$. Namely, locally the vectors $e^1(y), \ldots, e^n (y)$ constitute an orthonormal basis of the normal space to $\Lambda$ at $y$ and $y \mapsto e^j(y)$ is a smooth vector field. A moving orthonormal frame might not be globally defined but it is always defined in a neighborhood of a given point in $\Lambda$.

We can then define $\Phi$, a local parametrization from a neighborhood of $(y,0) \in \Lambda \times \mathbb R^n$ into a neighborhood of $y \in \mathbb R^m$, by
\[
\Phi  (y, z_1, \ldots, z_n)  : =  y + \sum_{i=1}^n z_i   e^i (y) ,
\]
and $(y, z_1, \ldots, z_n)$ will be referred to as Fermi coordinates. In this parametrization, the Euclidean metric
\begin{equation}\label{HUhisdfjbvbdjbv}
g_{\circ} : =  dx_1^2 + \ldots + dx_m^2 ,
\end{equation}
in $\mathbb R^m$, or more precisely  $\Phi^* g_\circ$, the pull back of $g_\circ$ by $\Phi$, is close to $\bar g$ the induced metric on $N\Lambda$.
The next Lemma gives a quantitative version of this statement.
\begin{lemma}
\label{appendilemma1}
In the above defined coordinates,
\begin{equation}
\label{appendilemma1bis}
\Phi^* g_\circ  : = \bar g + \sum_{i=1}^n z_i   (\mathring h^i + 2     \mathring \ell^i ) + \sum_{i,j=1}^n z_i   z_j   \mathring k^{ij},
\end{equation}
where the tensors $\mathring h^i$, $\mathring \ell^i$  and $\mathring k^{ij}$ acting on $T\Lambda$, have coefficients which are smooth functions on $\Lambda$.
\end{lemma}
\begin{proof}
We denote by
\[
(t_1, \ldots t_k)\longmapsto Y(t_1, \ldots, t_k) ,
\]
a parametrization of $\Lambda$ close to a given point $y_\bullet$ and, without  loss of generality, we assume that $Y(0) = y_\bullet$.

To keep notations short, we agree that $e^j \circ Y$ is also denoted by $e^j$. Hence
\[
X (t_1, \ldots, t_k, z_1, \ldots , z_n) : =  Y(t_1, \ldots, t_k) + \sum_{i=1}^n z_i   e^i(t_1,  \ldots, t_k),
\]
is a parametrization on $\mathbb R^m$ close to $y_\bullet$.  To compute the coefficients of the Euclidean metric in these coordinates, it is enough to compute
\[
\partial_{t_a} X \cdot \partial_{t_b} X ,  \qquad  \partial_{t_a} X \cdot \partial_{z_j} X \qquad \mbox{and} \qquad  \partial_{z_i} X \cdot \partial_{z_j} X.
\]
Observe that $\partial_{t_a} Y$ is a tangent vector to $\Lambda$ while $e^j$ is a normal  vector to $\Lambda$ and hence $\partial_{t_a} Y \cdot e^j \equiv 0$.  Using this, it is easy to check that
\[
\partial_{t_a} X \cdot \partial_{t_b} X =  \partial_{t_a} Y \cdot \partial_{t_b} Y + \sum_{i=1}^n z_i   (\partial_{t_a} Y \cdot \partial_{t_b} e^i + \partial_{t_b} Y \cdot \partial_{t_a} e^i  ) + \sum_{i,j=1}^n z_i   z_j    \partial_{t_a} e^i \cdot \partial_{t_b} e^j ,
\]
\[
\partial_{t_a} X \cdot \partial_{z_j} X =  \sum_{i=1}^n z_i   \partial_{t_a} e^i \cdot e^j \qquad \mbox{and} \qquad
\partial_{z_i} X \cdot \partial_{z_j} X = \sum_{i,j=1}^n    e^i \cdot e^j .
\]
We set
\[
\mathring h^i : =  \sum_{a,b=1}^k  (\partial_{t_a} Y \cdot \partial_{t_b} e^i + \partial_{t_b}  Y \cdot \partial_{t_a} e^i  )   dt_a   dt_b ,
\]
\[
\mathring k^{ij} : = \sum_{a,b=1}^k \partial_{t_a} e^i \cdot \partial_{t_b} e^j   dt_a   dt_b \qquad \mbox{and}
\qquad \mathring \ell^i : = \sum_{a=1}^k \sum_{j=1}^n  \partial_{t_a} e^i \cdot e^j   dt_a   dz_b .
\]
Observe that these are smooth functions defined on $\Lambda$.

With these notations at hand, we can write
\[
\begin{array}{rllll}
g_\circ  & = & \displaystyle \sum_{a,b =1}^k \left( \mathring g_{ab} + \sum_{i=1}^n z^i \mathring h^i_{ab} + \sum_{i,j =1}^n z_i   z_j   \mathring k^{ij}_{ab} \right)   dt_{a}   dt_b
\\[3mm]
& + & \displaystyle 2 \sum_{a =1}^k    \sum_{j=1}^n   \left( \sum_{i=1}^n z_i   \mathring \ell^i_{aj}  \right)   dt_a   dz_j + \sum_{i=1}^n  dz_i^2 ,
\end{array}
\]
so the proof of \eqref{appendilemma1bis} is completed.
\end{proof}

Recall that, if on a given manifold $M$ the metric  tensor is given  in local coordinates by
\[
g = \sum_{i,j =1}^m g_{ij}   dx_i   dx_j ,
\]
then the Laplace-Beltrami operator is given by
\[
\Delta_g : =  \frac{1}{\sqrt{{\rm det}    g }}   \sum_{i,j=1}^m \partial_{x_i} \left( \sqrt{{\rm det}    g } \,\,  g^{ij}   \partial_{x_j}   \cdot   \right) ,
\]
where $g^{ij}$ are the coefficients of the inverse of the matrix $(g_{ij})_{ij}$.

Using this formula, together with the expansion in Lemma~\ref{appendilemma1}, we get the~:
\begin{lemma}\label{frmicordinates}
In a tubular neighborhood of $\Lambda$, the Euclidean Laplacian $\Delta$ can be decomposed as
\[
\Delta : = \sum_{i=1}^m \partial_{x_i}^2 =  \Delta_{\bar g} + D ,
\]
where $\Delta_{\bar g} =  \Delta_{\mathring g} + \Delta_{g_z} $ denotes the Laplace-Beltrami operator on $N\Lambda$ for the metric $\bar g = \mathring g + g_z$, and $D$ is a second order differential operator  which, in Fermi coordinates,  can be expanded as
\[
D  = \sum_{i=1}^n z_i    D^{(2)}_{i} +  D^{(1)} ,
\]
where $D^{(2)}_{i}$ (respectively $D^{(1)}$) are second order (respectively first order) partial differential operators whose coefficients are smooth and bounded in some fixed tubular neighborhood of $\Lambda$.
\label{le:4.1}
\end{lemma}
\begin{proof}
The proof follows  from the result of Lemma \ref{appendilemma1} and the expression of the Laplacian in local coordinates.
\end{proof}

We can define, in a fixed tubular neighborhood of $\Lambda$,  the function $a$ by
\begin{equation}
{\rm dvol}_{g_\circ} = a \, {\rm dvol}_{\bar g}.
\label{eq:defa}
\end{equation}
Observe that $a$ is smooth and $a\equiv 1$ along $\Lambda$. Moreover, it follows from Lemma~\ref{appendilemma1}, that there exists a constant $C >0$ such that
\[
\left| a -1 \right| \leq C \, |z| ,
\]
in a tubular neighborhood of $\Lambda$.

\section{Construction of a sequence of approximate solutions}\label{sectgaggiusta}

We assume that $p\in(1, + \infty)$ if $n= 1, 2$ or $p\in (1, \frac{n+2}{n-2})$ if $n\geq 3$.  As in the introduction, we let  $U$ to be the unique positive radial solution of (\ref{problemradialguida}) and we set
\begin{equation}
\label{LINLALALLA}
L : = - \left(  \Delta + p  U^{p-1}\right) ,
\end{equation}
to be the linearized operator about $U$.  The spectrum of $L$ will be denoted by
\begin{equation}
\mu_0 < \mu_1 \leq \mu_2 \leq \mu_3 \leq \ldots
\end{equation}
and the corresponding eigenfunctions, which will be denoted by $\phi_j$, are normalized to have norm $1$ in $L^2(B_1^n)$. It is known (see Theorem 4.1 and Theorem 4.2 in  \cite{DGP}) that
\[
\mu_0 < 0 < \mu_1.
\]

We use the Fermi coordinates introduced in the previous section. According to \eqref{defspread1}, the function $\bar u_\epsilon$ only depends on $|z|$ which is nothing but the distance function from a point to $\Lambda$. We can write
\[
\Delta  \bar u_\epsilon + \bar u_\epsilon^p =  (\Delta_{g_z} \bar u_\epsilon + \bar u_\epsilon^p) + (\Delta_{\mathring g} + D )   \bar u_\epsilon .
\]
Since $\Delta_{\mathring g}  \bar u_\epsilon=0$ and since $\Delta_{g_z}  \bar u_\epsilon + \bar u_\epsilon^p=0$, we conclude that
\begin{equation}\label{starfidi}
\Delta  \bar u_\epsilon +\bar u_\epsilon^p =D    \bar u_\epsilon .
\end{equation}

As mentioned before, the idea is first to implement an iteration scheme to perturb $\bar u_\epsilon$ into a sequence of approximate solutions which are closer to being a genuine solution of our problem.  To do so, we write $u = \bar u_\epsilon +v$ and, making use of the result of Lemma~\ref{le:4.1} and \eqref{starfidi}, we rewrite the equation in (\ref{mainproblemepsilon}) as
\begin{equation}
- (\Delta_{g_z} + p   \bar u_\epsilon^{p-1} )   v =  E_\epsilon + K_\epsilon (v) + (\Delta_{\mathring g} + D)   v ,
\label{eq:4.11}
\end{equation}
where by definition
\[
E_\epsilon : =  \Delta \bar u_\epsilon + \bar u_\epsilon^p ,
\]
and
\[
K_\epsilon (v) : =   |\bar u_\epsilon + v|^{p} - \bar u_\epsilon^p -  p  \bar u_\epsilon^{p-1}   v .
\]

The iteration scheme we use is the following~: we set $v_{\epsilon,0} \equiv 0$ and, for all $i \geq 0$, we define inductively $v_{\epsilon,i+1}$ to be the solution of
\begin{equation}
\left\{
\begin{array}{rllll}
- (\Delta_{g_z} + p \bar u_\epsilon^{p-1} )  v_{\epsilon,i+1} & = &  E_\epsilon +  K_\epsilon (v_{\epsilon,i}) +  (\Delta_{\mathring g} + D)   v_{\epsilon,i}   & \mbox{in} \quad B_\epsilon( \Lambda) \\[3mm]
v_{\epsilon, i+1} & = &  0  & \mbox{on} \quad T_\epsilon( \Lambda) .
\end{array}
\right.
\label{eq:4.12}
\end{equation}
Observe that the functions are defined in $B_\epsilon (\Lambda)$ but the operator on the left hand side only depends on the variable normal to $\Lambda$, namely $\partial_{z_j}$. So, when we solve this equation, we only solve the equation in $B_\epsilon^n$, the ball of radius $\epsilon$ in $\mathbb{R}^n$ centered at the origin, and we consider the variable on $\Lambda$ as parameters. At each iteration, we loose two degrees of regularity in the variables belonging to $\Lambda$ but this is not a problem if we assume that  $\Lambda$ is differentiable enough, in fact this is where we need $\Lambda$ to be smooth if we want the sequence to be defined for all $i$ and $\Lambda$ should be regular enough if we just need a finite number of iteration. For sake of simplicity, we state and prove all results when $\Lambda$ is a smooth submanifold of $\mathbb R^m$, leaving the statement for the case where $\Lambda$ has finite smoothness to the reader.

To invert the left hand side of (\ref{eq:4.11}), we simply use  a scaling argument and the fact that, according to the result in \cite{DGP}, $0$ is not in the spectrum of the operator $L$ and hence this operator is invertible. In particular, if one wants to solve
\[
\left\{
\begin{array}{rllll}
- \left( \Delta  + p   \bar u_\epsilon^{p-1} \right) v & = & f \qquad & \mbox{in} \qquad B^n_\epsilon\\[3mm]
v & = & 0 \qquad & \mbox{on} \qquad \partial B_\epsilon^n ,
\end{array}
\right.
\]
one just considers $\tilde v (x) := v(\epsilon x)$ and $\tilde f (x) : = f(\epsilon x)$ which solve
\[
\left\{
\begin{array}{rllll}
- \left( \Delta + p  U^{p-1} \right)   \tilde v  & = &  \epsilon^2  \tilde f  \qquad & \mbox{in} \qquad B^n_1 \\[3mm]
\tilde v & = & 0 \qquad &  \mbox{on} \qquad \partial B_1^n .
\end{array}
\right.
\]
Standard elliptic estimates  for $\tilde v$ are available and the corresponding scaled estimates for the function $v$ follow at once.  Observe the gain of two powers of $\epsilon$ due to the presence of $\epsilon^2$ on the right hand side of the last equation.

We define
\begin{equation}
u_{\epsilon, i} : =  \bar u_\epsilon + v_{\epsilon, i}.
\label{eq:uepsiloni}
\end{equation}
We claim the following~:
\begin{proposition}
There exist constants $C>0$ and $\epsilon_0 >0$ such that, for all  $\epsilon \in (0,\epsilon_0)$ and for all $i \in \mathbb N$
\begin{equation}
\label{eq:4.11bis}
\| \Delta  u_{\epsilon,i} + u_{\epsilon,i}^p \|_{L^\infty (B_\epsilon (\Lambda))} \leq C   \epsilon^{i -\frac{p+1}{p-1}},
\end{equation}
and
\begin{equation}
\left\| \frac{v_{\epsilon,i}}{\bar u_\epsilon} \right\|_{L^\infty (B_\epsilon (\Lambda))}  + \epsilon   \left\| \frac{\partial_\epsilon v_{\epsilon,i}}{\bar u_\epsilon}  \right\|_{L^\infty (B_\epsilon (\Lambda))}  \leq C \epsilon .
\label{eq:4.12}
\end{equation}
\label{pr:4.1}
\end{proposition}

\begin{proof}  The proof is decomposed into a few steps each of which takes advantage of a particular property of the problem we are studying. To begin, we need to introduce the norms
\[
\begin{array}{rllll}
\| u\|_{\mathcal C^{0, \alpha}_\epsilon (B_\epsilon (\Lambda))} & : = &  \| u \|_{L^\infty (B_\epsilon (\Lambda))} + \displaystyle \sup_{(y,z),(y,z') \in \Omega_\epsilon (\Lambda)} \epsilon^{\alpha}   \frac{| u \circ F_1(y, z) - u \circ F_1 ( y , z')|}{|z-z'|^\alpha} ,
\end{array}
\]
and
\[
\begin{array}{rllll}
\| u\|_{\mathcal C^{2, \alpha}_\epsilon (B_\epsilon (\Lambda))} & : = &  \| u \|_{L^\infty (B_\epsilon (\Lambda))} +  \epsilon  \| \nabla_{g_z}  u \|_{L^\infty (B_\epsilon (\Lambda))} + \epsilon^2  \| \nabla_{g_z}^2 u \|_{L^\infty (B_\epsilon (\Lambda))}  \\[3mm]
& + & \displaystyle \sup_{(y,z),(y,z') \in \Omega_\epsilon (\Lambda)} \epsilon^{2+\alpha}   \frac{| \nabla_{g_z}^2 u \circ F_1(y, z) - \nabla_{g_z}^2 u \circ F_1 ( y , z')|}{|z-z'|^\alpha} .
\end{array}
\]

\noindent
{\bf Step 1.} According to \eqref{starfidi}, we have in local coordinates
\[
\Delta \bar u_\epsilon +\bar u_\epsilon^p = \sum_{i=1}^n z_i D^{(2)}_i \bar u_\epsilon + D^{(1)} \bar u_\epsilon  .
\]
As remarked earlier, the function $\bar u_\epsilon$ only depends on $|z|$ and hence we have
\[
 \| \bar u_\epsilon \|_{\mathcal C^{2, \alpha}_\epsilon (B_\epsilon (\Lambda))}  \leq C   \epsilon^{-\frac{2}{p-1}} .
\]
This follows from the the fact that $U \in \mathcal C^{2,\alpha}(B_1^n)$.

Taking advantage of the fact that the coefficients in $D^{(2)}_i$ and $D^{(1)}$ are smooth functions whose partial derivatives are bounded independently of $\epsilon$, we conclude that for all $\ell \in \mathbb N$, we have
\[
\epsilon \, \| \nabla_{\mathring g}^{\ell}   D^{(2)}_i \bar u_\epsilon \|_{\mathcal C^{0, \alpha}_\epsilon (B_\epsilon (\Lambda))}Ê+   \| \nabla_{\mathring g}^{\ell}  D^{(1)} \bar  u_\epsilon \|_{\mathcal C^{0, \alpha}_\epsilon (B_\epsilon (\Lambda))}Ê\leq C_{\ell}   \epsilon^{- \frac{p+1}{p-1}}  .
\]
for some constant $C_{\ell} >0$ which does not depend on $\epsilon \in (0,1)$. This implies that, for all $\ell \in \mathbb N$, there exists $C_{\ell} >0$ such that, for all $\epsilon \in (0,1)$
\[
\| \nabla_{\mathring g}^{ \ell}  \left( \Delta \bar u_\epsilon + \bar u_\epsilon^p\right) \|_{\mathcal C^{0,\alpha}_{\epsilon} (B_\epsilon (\Lambda))}\leq C_\ell   \epsilon^{-\frac{p+1}{p-1}}.
\]
Taking $\ell=0$, this already proves (\ref{eq:4.11bis}) when $i =0$.

\noindent
{\bf Step 2.}  To prove the first half of (\ref{eq:4.12}) when $i=1$, we use  the fact that
\[
- (\Delta_{g_z} + p   \bar u_\epsilon^{p-1}) v_{\epsilon, 1} =  \Delta \bar u_\epsilon + \bar u_\epsilon^p.
\]
Using the inverse of the operator $- (\Delta_{g_z} + p   \bar u_\epsilon^{p-1})$  and considering the variables on $\Lambda$ as parameters we get from standard elliptic estimates  that there exists $C_{\ell} >0$ such that, for all $\epsilon \in (0,1)$
\[
\|  v_{\epsilon,1} \|_{C^{2, \alpha}_\epsilon (B_\epsilon (\Lambda))} \leq C \epsilon^{1 - \frac{2}{p-1}} .
\]
We have obtained the first half of (\ref{eq:4.12}) when $i=1$.

\noindent {\bf Step 3.} We now derive some estimates for the partial derivatives of $v_{\epsilon,1}$ in the direction parallel to $\Lambda$. More precisely, we choose local coordinates $t_1, \ldots , t_k$ on $\Lambda$ and a cutoff function $\chi$ with  compact support where these coordinates are well defined. Observe that the operator
\[
\Gamma : f \mapsto \chi   \partial_{t_{i_1}} \ldots \partial_{t_{i_\ell}} f ,
\]
commutes with $-( \Delta_{g_z} + p   \bar u_\epsilon^{p-1})$ and hence
\[
- (\Delta_{g_z} + p   \bar u_\epsilon^{p-1}) Ê\Gamma v_{\epsilon, 1} =  \Gamma (\Delta \bar u_\epsilon + \bar u_\epsilon^p) .
\]
Moreover, since $v_{\epsilon, 1}$ vanishes on $T_\epsilon (\Lambda)$, so does $ \Gamma v_{\epsilon, 1}$. Since we have already estimated the right hand side of this equation, we can use the inverse of the operator $- (\Delta_{g_z} + p   \bar u_\epsilon^{p-1})$ and we get, for all $\ell \in \mathbb N$,
\begin{equation}
\label{estl}
\| \nabla_{\mathring g}^{ \ell}   v_{\epsilon, 1}Ê\|_{C^{2, \alpha}_\epsilon } \leq C_\ell \epsilon^{1 - \frac{2}{p-1}},
\end{equation}
where, as usual, $C_{\ell} >0$  does not depend on $\epsilon \in (0,1)$.


\noindent
{\bf Step 4.} To proceed, we  argue by induction. Taking the difference between the equation satisfied by $v_{\epsilon, i+1}$ and the equation satisfied by $v_{\epsilon, i}$ we get
\[
-(\Delta_{g_z} + p   \bar u_\epsilon)   (v_{\epsilon, i+1} - v_{\epsilon, i}) = K_\epsilon (v_{\epsilon, i}) - K_\epsilon (v_{\epsilon, i-1}) + (\Delta_{\mathring g} + D) (v_{\epsilon, i}- v_{\epsilon, i-1}) ,
\]
and one proves by induction that
\[
\|  \nabla_{\mathring g}^{ \ell} (v_{\epsilon,i+1} - v_{\epsilon, i} ) \|_{C^{2, \alpha}_\epsilon  (B_\epsilon (\Lambda))}  \leq C_\ell  \epsilon^{i +1 - \frac{2}{p-1}} ,
\]
where $C_{\ell} >0$  does not depend on $\epsilon \in (0,1)$. The proof uses the arguments already employed in Step 2 and Step 3. There is though one additional argument which is needed to estimate the nonlinear term $K_\epsilon (v)$. The key observation is that  $v_{\epsilon, i}$ vanishes on $\partial B_\epsilon (\Lambda)$ and if
\[
\| \nabla_{g_z}  v_{\epsilon, i}\|_{L^\infty (B_\epsilon (\Lambda))} \leq C   \epsilon^{- \frac{2}{p-1}},
\]
then, we have the pointwise estimate
\begin{equation}\label{gygysrrrstdgfnnnvbc}
|v_{\epsilon, i}| \leq \frac{1}{2}   \bar u_{\epsilon} ,
\end{equation}
provided $\epsilon$ is chosen close enough to $0$, since by Hopf boundary Lemma $U$ has non-zero normal derivative on $T_\epsilon (\Lambda)$. Hence, we are entitled to write
\[
K_\epsilon (v_{\epsilon, i}) = \bar u_\epsilon^{p}   \left( (1+ w_{\epsilon,i})^p -1-p   w_{\epsilon, i}\right),
\]
where
\[
w_{\epsilon, i} : =  \frac{v_{\epsilon, i}}{\bar u_{\epsilon}} ,
\]
takes values into $[-1/2, 1/2]$. In particular, we can use standard Taylor's expansion to evaluate the nonlinear term $K_\epsilon$. Details are left to the reader.
\end{proof}

Observe that \eqref{gygysrrrstdgfnnnvbc} also implies that, if $i\in \mathbb N$ is fixed, $u_{\epsilon, i} >0$ in $B_\epsilon (\Lambda)$ provided $\epsilon$ is chosen small enough.

\section{Analysis of the linearized operator about $u_{\epsilon, i}$}

We keep the notations of the previous section. In particular, $u_{\epsilon, i}$ denotes one of the approximate solutions which have been defined in \eqref{eq:uepsiloni}. In this section we are interested in the mapping properties of the  linearized operator about $u_{\epsilon, i}$, namely
\[
L_{\epsilon, i} : =  -  \left( \Delta + p   u_{\epsilon, i}^{p-1}\right) .
\]
We will exploit the fact that, in some sense to be made precise, this operator is close to the operator
\[
\mathcal L_\epsilon : =  -  \left( \Delta_{\bar g} + p  \bar u_\epsilon^{p-1} \right),
\]
whose eigenvalues are explicitly given by
\begin{equation}\label{defautoval}
\frac{\mu_i}{\epsilon^2} + \lambda_j ,
\end{equation}
where we recall that $(\mu_j)_{j \geq 0}$ are the eigenvalues of $L$ (defined in \eqref{LINLALALLA}) and where
\[
\lambda_0 =0 < \lambda_1 \leq \lambda_2\leq \ldots
\]
are the eigenvalues of $-\Delta_{\mathring g}$ on $\Lambda$.

There are some remarks which are straightforward but nevertheless very important. All rely on the fact that, as already mentioned,  $\mu_0 <0 <\mu_1$.

\begin{enumerate}

\item[(i)] The Morse index of  $\mathcal L_\epsilon$, which  is defined  to be the maximal dimension of  the subspaces of $H^1_0 (B_\epsilon (\Lambda))$ over which the quadratic form
\begin{equation}
\label{quad}
\mathcal Q_\epsilon  (v) : =  \int_{B_\epsilon (\Lambda)} \left( |\nabla_{\bar g} v|^2 - p    \bar u_\epsilon^{p-1}  v^2\right)   {\rm dvol}_{\bar g} ,
\end{equation}
is definite negative, is a decreasing function of $\epsilon$.  Observe that $\mathcal Q_\epsilon$ is defined using the volume form associated to the metric $\bar g$ on $N\Lambda$ and
\[
|\nabla_{\bar g} v|^2 =  |\nabla_{\mathring g} v|^2 +  |\nabla_{g_z} v|^2.
\]
By Weyl's formula (see for example \cite{He}) it is known that the number of eigenvalues of $-\Delta_{\mathring g}$ (counted with multiplicity), which are less than $\lambda >0$ is asymptotic  to $\lambda^{k/2}$ as $\lambda$ tends to infinity. Therefore, taking into account \eqref{defautoval} (see also (ii) here below),
 we get an estimate of the Morse index of $\mathcal L_\epsilon$, namely
\[
\mbox{\rm Index }    \mathcal L_{\epsilon} \sim  \epsilon^{-k}  .
\]

\item[(ii)] Observe that
\[
\frac{\mu_i}{\epsilon^2} + \lambda_j \geq \frac{\mu_1}{\epsilon^2} >0 ,
\]
for all $i \geq 1$ and all $j\geq 0$. In particular, the eigenfunctions associated to negative eigenvalues of $\mathcal L_\epsilon$ are of the form
\[
(y,z) \mapsto \phi_0 (z/ \epsilon)   \psi (y) ,
\]
where $\psi$ is an eigenfunction of $-\Delta_{\mathring g}$ and where we recall that $\phi_0$ is the eigenfunction of $L$ associated to $\mu_0$.

\item[(iii)]  The eigenvalues of $\mathcal L_\epsilon$ are monotone functions of $\epsilon$ and in fact
\[
\partial_\epsilon   \left( \frac{\mu_0}{ \epsilon^2} + \lambda_j \right) =  - 2   \frac{\mu_0}{\epsilon^3}.
\]
In particular, the Morse index of  $\mathcal L_\epsilon$ is a decreasing function of $\epsilon$.

\item[(iv)]  The spectrum of $\mathcal L_\epsilon$ contains $0$ if and only if
\[
\epsilon = \sqrt \frac{-\mu_0}{\lambda_j},
\]
and for this special values of $\epsilon$, the operator $\mathcal L_\epsilon$ (under $0$ Dirichlet boundary conditions) is not invertible.
\end{enumerate}

Having these remarks in mind, we now explain the argument we would use if we were to work with the operator $\mathcal L_\epsilon$ instead of $L_{\epsilon, i}$.

We define
\[
\mathcal Z_\epsilon : =  \{\epsilon >0   \, : \,  \exists j \in \mathbb N ,   \quad   \mu_0+ \epsilon^2   \lambda_j = 0 \} ,
\]
which corresponds to the set of $\epsilon$'s for which the operator
\[
\begin{array}{cccc}
H^1_0(B_\epsilon (\Lambda)) \cap H^2 (B_\epsilon (\Lambda)) & \to &L^2 (B_\epsilon (\Lambda)), \\[3mm]
w & \mapsto & \mathcal L_\epsilon  w,
\end{array}
\]
is not invertible. Now, if $\epsilon \notin \mathcal Z_\epsilon $,  we can estimate the norm  of the inverse of $\mathcal L_\epsilon$ by a constant times $1/\delta_\epsilon$ where $\delta_\epsilon$ is the distance from $0$ to the spectrum of $\mathcal L_\epsilon$, namely
\[
\delta_\epsilon : =  \min \left\{ \left| \frac{\mu_0}{ \epsilon^2} + \lambda_j \right|  \,  : \,  j \in \mathbb N \right\}.
\]

We fix $N \geq \max(2,k)$ and we define, for all $\epsilon$ such that $0<\epsilon <1$, the set
\[
S_{\epsilon, N} : = \{ \bar \epsilon \in (\epsilon, 2\epsilon)  \,  : \,   (\bar \epsilon - \epsilon^N , \bar\epsilon +\epsilon^N) \cap \mathcal Z_\epsilon = \varnothing \}.
\]
Property (i), which makes use of Weyl's asymptotic formula, implies in particular that    $\epsilon- \text{meas}(S_{\epsilon, N})$ cannot be larger than  a constant times $\epsilon^{N-k}$ and, for any $\bar \epsilon \in S_{\epsilon, N}$ we know from property (iii) (or from direct estimate) that the norm of the inverse of $\mathcal L_{\bar \epsilon}$ (defined as above) is bounded by a constant times $\epsilon^{3-N}$.

Therefore, if
\[
S_N : =  \bigcup_{\epsilon \in (0,1)} S_{\epsilon, N} ,
\]
then, for all $\epsilon \in S_N$, the norm of the inverse of $\mathcal L_\epsilon$ from $L^2 (B_\epsilon (\Lambda))$ into $L^2 (B_\epsilon (\Lambda))$ is controlled by a constant times $\epsilon^{3-N}$. Moreover, if $N-k \geq 2$ we have
\[
\lim_{\epsilon \to 0} \frac{1}{\epsilon^\alpha} \big( \epsilon - \mbox{meas }\big( S_N\cap (0,\epsilon)\big)\big) =0 ,
\]
provided $\alpha \in (1, N-k)$.

This is the argument we will try to adapt to the operator $L_{\epsilon, i}$. The main difficulty is that we will not be able to use separation of variables anymore, instead we will use the fact that the operators $L_{\epsilon, i}$ and $\mathcal L_\epsilon$ are close.

\subsection{Estimating the Morse index of $L_{\epsilon, i}$}

In this section, given $i \geq 0$, we recover partially Property (i) for the operator $L_{\epsilon, i}$.  We define the quadratic form associated to $L_{\epsilon, i}$ by
\begin{equation}
\mathcal Q_{\epsilon, i} (v) := \int_{B_\epsilon (\Lambda)} \left( | \nabla_{g_\circ} v|^2 -  p  u_{\epsilon,i}^{p-1}   v^2\right)   {\rm dvol}_{g_\circ} .
\end{equation}
Observe that the volume form is the one associated to the Euclidean metric $g_\circ$. Similarly,  the norm of the gradient of the function $v$ is computed using the Euclidean metric $g_\circ$. Comparing $\mathcal Q_{\epsilon, i}$ to the quadratic form $\mathcal Q_\epsilon$ which was defined in (\ref{quad}), we prove the
\begin{lemma}
\label{ND1}
Assume that $i\in \mathbb N$ is fixed. Then, there exists a constant $C>0$ such that
\[
\mbox{\rm Index }    L_{\epsilon,i} \leq C   \epsilon^{-k} ,
\]
for all $\epsilon >0$ close enough to $0$.
\end{lemma}
\begin{proof}
Let $v \in H^1_0 (B_\epsilon (\Lambda))$ which is normalized so that $\|v \|_{L^2(B_\epsilon (\Lambda))} =1$ and which satisfies $\mathcal Q_{\epsilon ,i} (v) \leq 0$. We want to estimate $\mathcal Q_\epsilon (v)$. Observe that the difference between $\mathcal Q_\epsilon$ and $\mathcal Q_{\epsilon, i}$ can be attributed to three different phenomena. First the difference between the square of the norm of the gradient of $v$ when the Euclidean metric or the product metric $\bar g$ are used, second the difference between the potentials $\bar u_\epsilon$ and $u_{\epsilon, i}$ and finally the difference between the volume forms   when the Euclidean metric or the product metric  $\bar g$ are used.

Using the result of Lemma~\ref{appendilemma1}, it is easy to check that
\[
\left| | \nabla_{g_\circ} v|^2 - | \nabla_{\bar g} v|^2\right| \leq C   \epsilon   | \nabla_{g_\circ} v|^2.
\]
Similarly, it follows from Proposition~\ref{pr:4.1}, that
\[
\left|  \bar u_\epsilon^{p-1} - u_{\epsilon, i}^{p-1} \right| \leq C   \epsilon   u_{\epsilon, i}^{p-1}.
\]
Finally, it follows once more from Lemma~\ref{appendilemma1} that the difference between the volume forms can be estimated in local coordinates by
\[
\left|  \sqrt{{\rm det}    \bar g} - \sqrt{{\rm det}    g_\circ } \right| \leq C    \epsilon \sqrt{{\rm det}   \bar g  }.
\]
Since $\mathcal Q_{\epsilon, i} (v) \leq 0$, we find that
\[
\begin{array}{rlll}
\mathcal Q_\epsilon (v) & = &  \mathcal Q_\epsilon (v)  - \mathcal Q_{\epsilon, i} (v) + \mathcal Q_{\epsilon, i} (v) \\[3mm]
&  \leq &  \mathcal Q_\epsilon (v)  - \mathcal Q_{\epsilon, i} (v) \\[3mm]
& \leq & \displaystyle C   \epsilon   \int_{B_\epsilon (\Lambda)} \left( | \nabla_{g_\circ} v|^2 +  u_{\epsilon,i}^{p-1} v^2\right) {\rm dvol}_{g_\circ} .
\end{array}
\]
Moreover,  $\mathcal Q_{\epsilon, i} (v) \leq 0$ also implies that
\[
 \int_{B_\epsilon (\Lambda)} | \nabla_{g_\circ} v|^2   {\rm dvol}_{g_\circ}  \leq  \int_{B_\epsilon (\Lambda)}  p  u_{\epsilon,i}^{p-1}   v^2   {\rm dvol}_{g_\circ} \leq  C  \epsilon^{-2}    \int_{B_\epsilon (\Lambda)}  v^2 {\rm dvol}_{\bar g} .
\]
Therefore, we have
\[
\mathcal Q_\epsilon (v)   \leq \frac{C}{\epsilon}   \int_{B_\epsilon (\Lambda)} v^2   {\rm dvol}_{\bar g}  = \frac{C}{\epsilon}.
\]
Using these, we see that the index of $L_{\epsilon, i}$ is bounded by the dimension of the space spanned by the eigenfunctions of $\mathcal L_{\epsilon}$ associated to eigenvalues less than or equal to $C/\epsilon$. Using Weyl's asymptotic formula and the explicit expression for the eigenvalues of $\mathcal L_\epsilon$, we conclude that the index of $L_{\epsilon, i}$ is bounded by a constant times $\epsilon^{-k}$ and this completes the proof of the result.
\end{proof}

\subsection{Decomposition of eigenfunctions associated to small eigenvalues}

Given $i\geq 0$, in this section, we recover Property (ii) for an operator close to $L_{\epsilon, i}$. Using the function $a$ defined in (\ref{eq:defa}), we set
\[
\tilde L_{\epsilon, i} : =  a \, L_{\epsilon, i}.
\]
Observe that $L_{\epsilon, i}$ is self-adjoint with respect to $L^2(B_\epsilon (\Lambda), g_\circ)$ while $\tilde L_{\epsilon, i}$ is self-adjoint with respect to $L^2(B_\epsilon (\Lambda), \bar g)$. Indeed, we have
\[
\begin{array}{rllll}
\displaystyle \int_{B_\epsilon (\Lambda)} v \, \tilde L_{\epsilon, i} w \, {\rm dvol}_{\bar g} & = & \displaystyle \int_{B_\epsilon (\Lambda)} v L_{\epsilon, i} w  \, {\rm dvol}_{g_\circ} \\[3mm]
& = & \displaystyle \int_{B_\epsilon (\Lambda)} w \,  L_{\epsilon, i} v \, {\rm dvol}_{g_\circ} \\[3mm]
& = & \displaystyle \int_{B_\epsilon (\Lambda)} w \, \tilde L_{\epsilon, i} v \, {\rm dvol}_{\bar g} .
\end{array}
\]
Also observe that the Morse index of $\tilde L_{\epsilon, i}$ is equal to the Morse index of $L_{\epsilon, i}$ since the two associated quadratic forms are equal.

Recall that we have denoted by  $\phi_0$ the eigenfunction of $- (\Delta + p \bar U^{p-1})$ associated to the eigenvalue $\mu_0 <0$ which is normalized to have $L^2$ norm equal to $1$. Observe that $\phi_0$ is radial and hence we can define
\[
\phi_{0,Ê\epsilon} (y,z) : =  \phi_0 (|z| / \epsilon) ,
\]
on $B_\epsilon (\Lambda)$.

Let $v$ be an eigenfunction of $\tilde L_{\epsilon, i}$ associated to the eigenvalue $\nu$. Hence
\[
\tilde L_{\epsilon, i}   v  = \nu   v ,
\]
in $B_\epsilon (\Lambda)$ and $v = 0$ on $T_\epsilon (\Lambda)$. We decompose
\begin{equation}
v (y,z) =  \phi_{0, \epsilon} (z)   \psi(y) + \bar v (y,z) ,
\label{eq:decomp}
\end{equation}
where $\psi$ is a function defined on $\Lambda$ and
\[
\int_{B_\epsilon (\Lambda)} \bar v  \phi_{0, \epsilon}   h   {\rm dvol}_{\bar g} =0 ,
\]
for any $h \in L^2(\Lambda)$. Observe that the orthogonality condition is expressed using the metric  $\bar g$ and not $g_\circ$. As usual, we identify $B_\epsilon (\Lambda)$ with a tubular neighborhood of the zero section in $N\Lambda$. We have the~:
\begin{lemma}
There exists constants $C_0, C >0$ such that, if $v$  is a solution of $\tilde L_{\epsilon, i}   v  = \nu   v$ which is decomposed as in (\ref{eq:decomp}) and if we further assume that $$
\displaystyle \nu \leq \frac{C_0}{\epsilon^2},
$$
then
\begin{equation}
\int_{B_\epsilon (\Lambda)} \left( |\nabla_{\bar g} \bar v|^2 + \frac{1}{\epsilon^2} \, \bar v^2  \right) \,  {\rm dvol}_{\bar g}  \leq  \frac{C}{\epsilon}  \, \int_{B_\epsilon (\Lambda)} v^2    {\rm dvol}_{\bar g}.
\label{eq:5.2}
\end{equation}
\label{le:5.2}
\end{lemma}
\begin{proof}
For notational convenience, we set $v_0 (y,z) =  \phi_{0, \epsilon} (z)   \psi(y)$ so that $v = v_0+ \bar v$.  In the proof, one has to be careful since there are two different metrics which are used in $B_\epsilon (\Lambda)$. The first metric is the Euclidean metric $g_\circ$ with respect to which $L_{\epsilon, i}$ is self-adjoint and the second metric is $\bar g$ with respect to which $\mathcal L_\epsilon$ is self-adjoint.

\noindent {\bf Step 1.} We exploit the fact that $\tilde L_{\epsilon, i}   v = \nu   v$ by multiplying this equation by $v$ and integrating the result over $B_\epsilon (\Lambda)$ to find
\begin{equation}
\int_{B_\epsilon (\Lambda)} \left( |\nabla_{g_\circ} v|^2 - p   u_{\epsilon, i}^{p-1}   v^2\right)   {\rm dvol}_{g_\circ} = \nu   \int_{B_\epsilon (\Lambda)} v^2   {\rm dvol}_{\bar g} .
\label{eq:d1}
\end{equation}
Since $u_{\epsilon, i}^{p-1} \leq C   \epsilon^{-2}$ and since the volume forms associated to $g_\circ$ and $\bar g$ are equivalent in a neighborhood of $\Lambda$, we get the estimate
\begin{equation}
\int_{B_\epsilon (\Lambda)} |\nabla_{g_\circ} v|^2   {\rm dvol}_{g_\circ} \leq  \left( \nu + \frac{C}{\epsilon^2}\right)    \int_{B_\epsilon (\Lambda)} v^2   {\rm dvol}_{\bar g}.
\label{eq:d2}
\end{equation}

\noindent {\bf Step 2.}  Next, we exploit the fact that the eigenvalues of $\mathcal L_\epsilon$ are explicitly known and the orthogonal decomposition of $v$ implies that
\[
\frac{\mu_1}{\epsilon^2} \, \int_{B_\epsilon (\Lambda)} \, \bar v^2  \,  {\rm dvol}_{\bar g}  \leq \int_{B_\epsilon (\Lambda)} \left( |\nabla_{\bar g} \bar v|^2 - p  \bar u_\epsilon^{p-1}   \bar v^2  \right)      {\rm dvol}_{\bar g} .
\]
Using the fact that $p  \bar u_\epsilon^{p-1}  \leq C \epsilon^{-2}$, we conclude that there exists a constant $C_1>0$ such that
\[
C_1 \, \int_{B_\epsilon (\Lambda)} \left( |\nabla_{\bar g} \bar v|^2 + \frac{1}{\epsilon^2} \, \bar v^2  \right) \,  {\rm dvol}_{\bar g}  \leq \int_{B_\epsilon (\Lambda)} \left( |\nabla_{\bar g} \bar v|^2 - p  \bar u_\epsilon^{p-1}   \bar v^2  \right)      {\rm dvol}_{\bar g} .
\]
Since $\mathcal L_{\epsilon} v_0$ is $L^2 (B_\epsilon (\Lambda) , \bar g)$ orthogonal to $\bar v$, we conclude that
\[
\displaystyle C_1 \, \int_{B_\epsilon (\Lambda)} \left( |\nabla_{\bar g} \bar v|^2 + \frac{1}{\epsilon^2} \, \bar v^2  \right) \,  {\rm dvol}_{\bar g}  \leq \int_{B_\epsilon (\Lambda)} \left( \nabla_{\bar g} v \cdot \nabla_{\bar g} \bar v  - p  \bar u_\epsilon^{p-1}    v  \bar v  \right)      {\rm dvol}_{\bar g} .
\]

\noindent {\bf Step 3.}   As in the proof of Lemma~\ref{ND1}, we can replace the metric $\bar g$ by the metric $g_\circ$ and the function $\bar u_{\epsilon}$ by $u_{\epsilon, i}$ on the right hand side and, using the results of Lemma~\ref{appendilemma1} and Proposition~\ref{pr:4.1}, we conclude that
\[
\begin{array}{rllll}
\displaystyle C_1 \, \int_{B_\epsilon (\Lambda)} \left( |\nabla_{\bar g} \bar v|^2 + \frac{1}{\epsilon^2} \, \bar v^2  \right) \,  {\rm dvol}_{\bar g}  & \leq & \displaystyle \int_{B_\epsilon (\Lambda)} \left( \nabla_{g_{\circ}} v \cdot \nabla_{g_\circ} \bar v  - p  u_{\epsilon, i}^{p-1}    v  \bar v  \right)      {\rm dvol}_{g_\circ} \\[3mm]
& + & \displaystyle C \epsilon    \int_{B_\epsilon (\Lambda)} \left( |\nabla_{\bar g} \bar v|^2+  |\nabla_{\bar g}  v|^2 \right)   {\rm dvol}_{\bar g}  \\[3mm]
& + & \displaystyle  \frac{C}{\epsilon}    \int_{B_\epsilon (\Lambda)} ( \bar v^2 + v^2 )   {\rm dvol}_{\bar g} .
\end{array}
\]
Since $\tilde L_{\epsilon, i} v = \nu   v$, we conclude that
\[
\begin{array}{rllll}
\displaystyle C_1 \, \int_{B_\epsilon (\Lambda)} \left( |\nabla_{\bar g} \bar v|^2 + \frac{1}{\epsilon^2} \, \bar v^2  \right) \,  {\rm dvol}_{\bar g}  & \leq & \displaystyle \nu \, \int_{B_\epsilon (\Lambda)} \bar v^2  \,  {\rm dvol}_{\bar g} \\[3mm]
& + & \displaystyle C \epsilon    \int_{B_\epsilon (\Lambda)} \left( |\nabla_{\bar g} \bar v|^2+  |\nabla_{\bar g}  v|^2 \right)   {\rm dvol}_{\bar g}  \\[3mm]
& + & \displaystyle  \frac{C}{\epsilon}    \int_{B_\epsilon (\Lambda)} ( \bar v^2 + v^2 )   {\rm dvol}_{\bar g} .
\end{array}
\label{eq:D3}
\]
On the right hand side, the terms in $\bar v$ can be absorbed in the left hand side provided $\epsilon$ is chosen small enough and $\nu \leq C_1 / (2\epsilon^2)$. We conclude that
\[
\begin{array}{rllll}
\displaystyle \int_{B_\epsilon (\Lambda)} \left( |\nabla_{\bar g} \bar v|^2 + \frac{1}{\epsilon^2} \, \bar v^2  \right) \,  {\rm dvol}_{\bar g}  & \leq & \displaystyle C \epsilon    \int_{B_\epsilon (\Lambda)} \left(  |\nabla_{\bar g}  v|^2  +  \frac{1}{\epsilon^2}  \,  v^2 \right)  {\rm dvol}_{\bar g} .
\end{array}
\]
This together with (\ref{eq:d2}) implies that that
\[
\displaystyle \int_{B_\epsilon (\Lambda)} \left( |\nabla_{\bar g} \bar v|^2 + \frac{1}{\epsilon^2} \, \bar v^2  \right) \,  {\rm dvol}_{\bar g}  \leq \displaystyle \frac{C}{\epsilon}    \int_{B_\epsilon (\Lambda)}   v^2   {\rm dvol}_{\bar g} ,
\]
and this completes the proof of the lemma.
\end{proof}

\subsection{Exploiting Kato's result}\label{KATOKATO}

In this section, we estimate the rate of change of eigenvalues of $\tilde L_{\epsilon, i}$ as $\epsilon$ varies. In other words, we obtain for the operator $\tilde L_{\epsilon, i}$ a result, which is close to the Property (iii) which was straightforward for the operator $\mathcal L_\epsilon$.

Let us explain the proof in the case where $\nu : =  \nu(\epsilon)$ is a simple eigenvalue for $\tilde L_{\epsilon, i}$. It is known that in this case $\nu$ depends smoothly on $\epsilon$. To proceed, we need to work with functions defined on a fixed domain which does not depend on $\epsilon$. Hence, we parameterize $B_{\epsilon} (\Lambda)$ using
\[
\begin{array}{rllll}
F_\epsilon : &  \Omega_1(\Lambda)  &  \to  & B_{\epsilon} (\Lambda)\\[3mm]
  & (y, z) &  \mapsto & y + \epsilon z ,
\end{array}
\]
where $\Omega_1(\Lambda) := \{ (y,z) \in N\Lambda \, : \, |z|< 1\}$. Observe that
\[
F^*_\epsilon \bar g  = \mathring g + \epsilon^2 \, d z^2 .
\]
We define the operator  $\hat L_{\epsilon, i}$ by
\[
\hat L_{\epsilon, i} (v \circ F_\epsilon)  = (\tilde L_{\epsilon, i} v) \circ F_\epsilon.
\]

Let $v := v(\epsilon)$ be the eigenfunction of $\tilde L_{\epsilon, i}$ associated to $\nu = \nu (\epsilon)$. By definition of $\hat L_{\epsilon, i}$, we have
\begin{equation}
\hat L_{\epsilon, i} w = \nu \, w,
\label{eq:eig}
\end{equation}
where $w : =  v \circ F_\epsilon$.  Without loss of generality we can assume that $w$ depends smoothly on $\epsilon$ and is normalized so that
\[
\int_{\Omega_1(\Lambda)} w^2 \, {\rm dvol}_{\bar g} =1.
\]
Differentiation of (\ref{eq:eig}) with respect to $\epsilon$, yields
\[
\hat L_{\epsilon, i} (\partial_\epsilon w) +   \left(\partial_\epsilon \hat L_{\epsilon, i}\right)  \, w = (\partial_\epsilon \nu ) \, w + \nu \, (\partial_\epsilon w).
\]
Multiplying this equation by $w$ and integrating over $\Omega_1(\Lambda)$, we get
\[
\partial_\epsilon \nu  = \int_{\Omega_1(\Lambda)}  w \, (\partial_\epsilon \hat L_{\epsilon, i})  \, w \, {\rm dvol}_{\bar g}.
\]

When the eigenspaces are not simple, we can interpret $\partial_\epsilon \nu$ as a set-valued function which takes into account the possibility that the eigenvalue splits into a number of separate eigenvalues (see \cite{K} and \cite{Cox}). The estimate for the elements of this set of derivatives is given by
\[
 \partial_\epsilon \nu \in\left \{  \int_{\Omega_1 (\Lambda)} w \, (\partial_\epsilon \hat L_{\epsilon, i})  \, w \, {\rm dvol}_{\bar g} \, : \,  \hat L_{\epsilon,i}  w =  \nu \,  w \quad \text{and}\quad \int_{\Omega_1(\Lambda)} w^2 \, {\rm dvol}_{\bar g} =1 \right\}.
\]

We have the~:
\begin{lemma}
\label{KATOsss}
There exists a constant $C_2 >0$ such that if $i \in \mathbb N$ is fixed and if $\nu$ is an eigenvalue of $\tilde L_{\epsilon,i}$ such that
\[
\nu \leq \frac{C_1}{\epsilon^2},
\]
where $C_1$ is the constant defined in Lemma~\ref{le:5.2}, then
\[
\partial_\epsilon \nu \geq  \frac{C_2}{\epsilon^3} ,
\]
for all $\epsilon$ small enough.
\end{lemma}
\begin{proof} We use the decomposition of the Laplacian which was given in Lemma~\ref{frmicordinates} together with the estimates (\ref{eq:4.12})  for $v_{\epsilon, i}$ and $\partial_\epsilon v_{\epsilon,i}$ which were given in Proposition~\ref{pr:4.1}. With little work we conclude that
\[
\begin{array}{rlllll}
\displaystyle \int_{\Omega_1 (\Lambda)} w \, (\partial_\epsilon \tilde L_{\epsilon, i})  \, w \, {\rm dvol}_{\bar g} & \geq &\displaystyle   - \frac{2}{\epsilon^3} \int_{\Omega_1 (\Lambda)} \left( |Ê\nabla_{g_z} w|^2 - p \, U^{p-1} \, w^2\right)  \, {\rm dvol}_{\bar g} \\[3mm]
& - & \displaystyle  C \, \int_{\Omega_1 (\Lambda)} \left( |\nabla_{\mathring g} w|^2 + \frac{1}{\epsilon^2} \, (| \nabla_{g_z} w|^2 + w^2)\right) \, {\rm dvol}_{\bar g} .
\end{array}
\]
Indeed, the expression of  $\tilde L_{\epsilon, i}$ in local coordinates $t : = (t_1, \ldots t_k)$ on $\Lambda$ and $z:= (z_1, \ldots, z_n)$ on the normal section can be written as
\[
\tilde L_{\epsilon, i} = - \tilde a \left( \Delta_{\mathring g} + \frac{1}{\epsilon^2} \, \left( \Delta_{g_z} + p ( U+ \epsilon^{\frac{2}{p-1}} \tilde v_{\epsilon, i})^{p-1} \right) \right)  + \epsilon \, \tilde D^{(2)}  + \tilde D^{(1)}  ,
\]
where $\tilde a := a \circ F_\epsilon$, $\tilde v_{\epsilon, i} : =  v_{\epsilon, i} \circ F_\epsilon$ and where $\tilde D^{(2)}$ (respectively $\tilde D^{(1)}$) is a second order (respectively first order) partial differential operator in $\partial_{t_j}$ and $\epsilon^{-1}\, \partial_{z_i}$ with smooth coefficients in $\Omega_1(\Lambda)$.

Differentiating with respect to $\epsilon$ and using  (\ref{eq:4.12}) we conclude that
\[
\partial_\epsilon \tilde L_{\epsilon, i} = \frac{2}{\epsilon^3}\,  \left( \Delta_{g_z} + p U^{p-1}\right)  +  \hat D^{(2)}  + \hat D^{(1)} + \frac{1}{\epsilon^2} \, \hat D^{(0)}  ,
\]
where $\hat D^{(j)}$ is a j-th order  partial differential operator in $\partial_{t_j}$ and $\epsilon^{-2}\, \partial_{z_i}$ with smooth coefficients in $\Omega_1(\Lambda)$.

Now, we decompose the eigenfunctions $w$ satisfying $\hat L_{\epsilon,i}  w =  \nu \,  w$ into $w = w_0 + \bar w$ as in  Lemma~\ref{le:5.2} to get
\[
\begin{array}{rlllll}
\displaystyle \int_{\Omega_1 (\Lambda)} w \, (\partial_\epsilon \tilde L_{\epsilon, i})  \, w \, {\rm dvol}_{\bar g} & \geq &\displaystyle   - \frac{2}{\epsilon^3} \int_{\Omega_1 (\Lambda)} \left( |Ê\nabla_{g_z} w_0|^2 - p \, U^{p-1} \, w_0^2\right)  \, {\rm dvol}_{\bar g} \\[3mm]
&-  &\displaystyle   \frac{2}{\epsilon^3} \int_{\Omega_1 (\Lambda)} \left( |Ê\nabla_{g_z} \bar w|^2 - p \, U^{p-1} \, \bar w^2\right)  \, {\rm dvol}_{\bar g} \\[3mm]
& - & \displaystyle  C \, \int_{\Omega_1 (\Lambda)} \left( |\nabla_{\mathring g} w|^2 + \frac{1}{\epsilon^2} \, (| \nabla_{g_z} w|^2 + w^2)\right) \, {\rm dvol}_{\bar g} .
\end{array}
\]
Since
\[
\int_{\Omega_1 (\Lambda)} \left( |Ê\nabla_{g_z} w_0|^2 - p \, U^{p-1} \, w_0^2\right)  \, {\rm dvol}_{\bar g}  = \mu_0 \, \int_{\Omega_1 (\Lambda)}  w_0^2  \, {\rm dvol}_{\bar g} ,
\]
we conclude, using the estimate (\ref{eq:5.2}) and (\ref{eq:d1}) in the proof of Lemma~\ref{le:5.2} that
\[
\begin{array}{rlllll}
\displaystyle \int_{\Omega_1 (\Lambda)} w \, (\partial_\epsilon \tilde L_{\epsilon, i})  \, w \, {\rm dvol}_{\bar g} & \geq &\displaystyle   - \frac{2\mu_0 }{\epsilon^3} \int_{\Omega_1 (\Lambda)}  w_0^2 \, {\rm dvol}_{\bar g} -  \displaystyle  \frac{C}{\epsilon^2} \, \int_{\Omega_1 (\Lambda)} w^2  \, {\rm dvol}_{\bar g}.
\end{array}
\]
Since
\[
\int_{\Omega_1 (\Lambda)}  w_0^2 \, {\rm dvol}_{\bar g} = \int_{\Omega_1 (\Lambda)}  w^2 \, {\rm dvol}_{\bar g} - \int_{\Omega_1 (\Lambda)}  \bar w^2 \, {\rm dvol}_{\bar g},
\]
we can again use the result of Lemma~\ref{le:5.2} to conclude that
\[
\begin{array}{rlllll}
\displaystyle \int_{\Omega_1 (\Lambda)} w \, (\partial_\epsilon \tilde L_{\epsilon, i})  \, w \, {\rm dvol}_{\bar g} & \geq &\displaystyle   - \left( \frac{2\mu_0 }{\epsilon^3}   + \frac{C}{\epsilon^2}\right)  \, \int_{\Omega_1 (\Lambda)} w^2  \, {\rm dvol}_{\bar g} .
\end{array}
\]
One can then choose the constant $C_2 >0$ to be any number $C_2 < -2 \mu_0$.
\end{proof}

\section{Uniform Invertibility of $L_{\epsilon, i}$}

We now have all the ingredients to apply to the operator $L_{\epsilon, i}$, the strategy which was outlined at the beginning of  \S 5 for the operator $\mathcal L_\epsilon$.

We fix $i \geq 0$ and $\epsilon >0$. We denote by $\Sigma_{\epsilon,i}$ the spectrum of $\tilde L_{\epsilon, i}$ and we define
\[
\mathcal Z_{\epsilon, i}  : =  \{ \epsilon >0   \,  :  \,   0 \in \Sigma_{\epsilon, i} \} ,
\]
which corresponds to the set of $\epsilon$'s for which the operator
\[
\begin{array}{cccllll}
H^1_0(B_\epsilon (\Lambda)) \cap H^2 (B_\epsilon (\Lambda)) &  \to & L^2 (B_\epsilon (\Lambda)), \\[3mm]
 w & \mapsto & \tilde L_{\epsilon ,i} w,
 \end{array}
\]
is not invertible. It is standard that, if $\epsilon \notin \mathcal Z_{\epsilon, i}$,  one can estimate the norm of the inverse of $\tilde L_{\epsilon, i}$ by a constant times $1/\delta_{\epsilon, i}$ where $\delta_{\epsilon, i}$ is the distance from $0$ to the spectrum of $\tilde L_{\epsilon, i}$, namely
\[
\delta_{\epsilon ,i} : =  \min \left\{ \left| \nu \right|  \,  : \,  \nu  \in \Sigma_{\epsilon, i} \right\}.
\]

We fix $N \geq \max(2,k)$ and we define, for all $\epsilon$ such that  $0<\epsilon <1$, the set
\[
S_{\epsilon, i, N} : = \{ \bar \epsilon \in (\epsilon, 2\epsilon)  \, : \, (\bar\epsilon - \epsilon^N , \bar\epsilon +\epsilon^N) \cap \mathcal Z_{\epsilon, i} = \varnothing \}.
\]
The result of Lemma~\ref{ND1}, implies that   $\epsilon- \text{meas } (S_{\epsilon,i,  N})$ cannot be larger than  a constant times $\epsilon^{N-k}$ and, for any $\bar \epsilon \in S_{\epsilon, i, N}$ we know from Lemma~\ref{KATOsss} that the norm of the inverse of $\tilde L_{\epsilon, i}$ (defined as above) is bounded by a constant times $\epsilon^{3-N}$ and, since $\tilde L_{\epsilon, i} = a \, L_{\epsilon, i}$ where $a$ is bounded away from $0$, a similar property holds for $L_{\epsilon, i}$.

Therefore, if
\[
S_{i,N} : =  \bigcup_{\epsilon \in (0,1)} S_{\epsilon, i, N} ,
\]
then, for all $\epsilon \in S_{i,N}$, the norm of the inverse of $L_{\epsilon, i}$ defined from $L^2 (B_\epsilon (\Lambda))$ into $L^2 (B_\epsilon (\Lambda))$, is controlled by a constant times $\epsilon^{3-N}$.  Moreover, if $N-k \geq 2$ we have
\[
\lim_{\epsilon \to 0} \frac{1}{\epsilon^\alpha} \big( \epsilon - \mbox{meas } \big( S_{i,N}\cap(0\,,\,\epsilon) \big) \big) = 0 ,
\]
provided $\alpha < N-k$.

\begin{definition}
We define $C^1_0(B_\epsilon (\Lambda))$ to be the subspace of $C^1(B_\epsilon (\Lambda))$ spanned by functions which vanish on $T_\epsilon (\Lambda)$.
\end{definition}

Now it is enough to invoke Schauder's estimates to estimate the norm of the inverse of $L_{\epsilon, i}$ when defined from $C^0(B_\epsilon (\Lambda))$ into $C^1_0 (B_\epsilon (\Lambda))$ (here we use the Euclidean metric to estimate the norm of the partial derivatives of functions). Using Schauder's estimates to control the norm of the inverse of $L_{\epsilon, i}$ between $C^\ell$ spaces starting from the knowledge of its norm between Lebesgue spaces, we  loose a few powers of $\epsilon$, say $\epsilon^{-N_0}$, where $N_0$ only depends on the dimension $m$. We have proven the~:
\begin{lemma}
Given $i\geq 0$ and $N\geq k+2$, there exist $S_{i,N} \subset (0,+\infty)$  and $N_0 \in \mathbb N$ such that, for all $\epsilon \in S_{i,N}$ the operator $L_{\epsilon, i}$ is invertible and the norm of its inverse defined from $C^0(B_\epsilon (\Lambda))$ into $C^1_0 (B_\epsilon (\Lambda))$ is bounded by a constant times $\epsilon^{3-N-N_0}$. Moreover
\[
\lim_{\epsilon \to 0} \frac{1}{\epsilon^\alpha} \big( \epsilon - \mbox{meas } \big( S_{i,N}\cap(0\,,\,\epsilon) \big) \big) = 0 ,
\]
provided $\alpha < N-k$.
\label{le:gi}
\end{lemma}

\section{A perturbation argument and the proof of Theorem~\ref{maintheoremenunciato}}

Thanks to the previous analysis, we can now give the proof of Theorem~\ref{maintheoremenunciato}. We keep the notations introduced in the previous section.

As already mentioned, we perturb the approximate solution $u_{\epsilon, i}$. Therefore, we look for a solution $u = u_{\epsilon, i} + v$ so that the equation to solve can be written as
\begin{equation}
L_{\epsilon, i}  v = E_{\epsilon, i} + K_{\epsilon, i} (v) ,
\label{eq:tosolve}
\end{equation}
where by definition
\[
E_{\epsilon, i} : =  \Delta u_{\epsilon, i} + u_{\epsilon, i}^p ,
\]
and
\[
K_{\epsilon, i} (v) : = | u_{\epsilon, i} +v|^p- u_{\epsilon, i}^p-p\, u_{\epsilon, i}^{p-1} \, v .
\]
As in the proof of Proposition~\ref{pr:4.1}, it will be convenient to observe that
\[
K_{\epsilon, i} (v) : = u_{\epsilon, i}^p \, \left(  \left| 1 + \frac{w}{u_{\epsilon, i}} \right|^p- 1- p \, \frac{w}{u_{\epsilon, i}} \right) ,
\]
so that one can use Taylor's expansion to evaluate the nonlinear terms, provided $w/u_{\epsilon, i}$ is small enough.

We fixe $\alpha>1$ as in the statement of Theorem~\ref{maintheoremenunciato} and $N \geq k+2$. Then, we choose $i > 2 (N+ N_0)-3$ and $M$ such that
\begin{equation}
\label{c1}
i  + 2-N-N_0  -\frac{2}{p-1} >  M ,
\end{equation}
and
\begin{equation}
\label{c2}
M > N+N_0 - 1 -  \frac{2}{p-1} .
\end{equation}

According to (\ref{eq:4.11bis}), we have
\[
\|  E_{\epsilon, i}  \|_{L^\infty (B_\epsilon (\Lambda))} \leq C\, \epsilon^{i-\frac{p+1}{p-1}} ,
\]
and we can use the result of Lemma~\ref{le:gi} to evaluate the norm of $L^{-1}_{\epsilon, i}$, the inverse of $L_{\epsilon, i}$, by
\[
\| L^{-1}_{\epsilon, i}\|_{C^0 \to C^1} \leq C \,  \epsilon^{3-N-N_0}.
\]

Recall that $C^1_0(B_\epsilon (\Lambda))$ denotes the subspace of $C^1(B_\epsilon (\Lambda))$ spanned by functions which vanish on $T_\epsilon (\Lambda)$. Now, assume that $v \in  C^1_0(B_\epsilon (\Lambda))$ satisfies $\| v\|_{C^1_0(B_\epsilon (\Lambda))} \leq \epsilon^{M}$ where $M$ is fixed as above.
Since $v$ vanished on $T_\epsilon (\Lambda)$ and since the gradient of $v$ is bounded by $\epsilon^M$ we conclude that $|v/u_{\epsilon, i}|\leq 1/2$ for all $\epsilon >0$ small enough. Hence we find
\[
\| K_{\epsilon, i} (v_2) - K_{\epsilon, i} (v_1) \|_{L^\infty (B_\epsilon (\Lambda))} \leq C \,  \epsilon^{M - 2 \frac{p-2}{p-1}} \, \| v_2 -v_1\|_{L^\infty (B_\epsilon (\Lambda))} ,
\]
for all $v_2, v_1 \in C^1_0(B_\epsilon (\Lambda))$ such that  $\| v_i \|_{C^1_0(B_\epsilon (\Lambda))} \leq \epsilon^{M}$.

We can then rephrase the solvability of (\ref{mainproblemepsilon}) as a fixed point problem
\[
v : = L^{-1}_{\epsilon, i} \left(   E_{\epsilon, i} + K_{\epsilon, i} (v)  \right),
\]
and apply a standard fixed point argument for contraction mapping in
\[
\left\{ v \in C^1_0(B_\epsilon (\Lambda)) \quad : \quad \| v\|_{C^1(B_\epsilon (\Lambda))} \leq \epsilon^{M} \right\} .
\]
The choice of $M$ implies that we have a contraction mapping (this fact uses (\ref{c2})) from this set into itself (this fact uses (\ref{c1})). This completes the proof of Theorem~\ref{maintheoremenunciato}.

\section{A Pohozaev type argument and the proof of Theorem \ref{nonexistence}} In this last section, we give a proof of Theorem~\ref{nonexistence} using a refined version of the celebrated technique introduced in \cite{Poh} are usually referred to as {\em Pohozaev identity}. We exploit an appropriate use of test functions, and explicit computations carried out in Fermi coordinates.

We start with a general result in the following~:
\begin{lemma}
Assume $D\subseteq \mathbb{R}^m$  is an open set, and let  $\phi\in C^2(D)$, then
\[
\begin{array}{rllll}
	\displaystyle \mbox{div} \left( (\nabla u \cdot \nabla \phi) \, \nabla u -\left( \frac{1}{2} \, |\nabla u|^2 - \frac{1}{p+1} u^{p+1} \right) \, \nabla \phi + \frac{1}{p+1} u \, \Delta \phi \, \nabla u \right) &  &  \\[3mm]
	\displaystyle + \left( \frac{1}{n} \, \Delta \phi \, |\nabla u|^2 - \nabla^2 \phi (\nabla u , \nabla u) \right) + \left( \frac{n-2}{2n} - \frac{1}{p+1}\right) \, |\nabla u|^2   \, \Delta \phi &  &  \\[3mm]
	\displaystyle - \frac{1}{p+1} u \, \nabla u \cdot \nabla \Delta \phi & = &  0 ,
\end{array}
\]
in $D$, provided $u$ is a classical solution of $\Delta u + u^p=0$ in $D$.
\end{lemma}
\begin{proof} Multiplying the equation $\Delta u + u^p=0$ by $u$, we get
\begin{equation}
	\mbox{div} \, ( u \, \nabla u) = |\nabla u|^2 - u^{p+1}.
\label{eq:PI-1}
\end{equation}
Next, we multiply the equation $\Delta u + u^p =0$ by $\nabla \phi \cdot \nabla u$ to get after some simple manipulation
\[
\begin{array}{rlll}
	\displaystyle \mbox{div} \left( (\nabla u \cdot \nabla \phi) \, \nabla u + \frac{1}{p+1} u^{p+1} \, \nabla \phi \right) - \nabla^2 \phi (\nabla u , \nabla u)  - \frac{1}{p+1} u^{p+1}  \, \Delta \phi  &    &  \\[3mm]
	\displaystyle - \frac{1}{2} \, \nabla \phi \cdot \nabla ( |\nabla u|^2)   & = & 0.
\end{array}
\]
Trying to write the last term on the left hand side as a divergence and correcting, we conclude that
\begin{equation}
\label{hfggghsjbvcvbcbvnnvmmc}
\begin{array}{rllll}
	\displaystyle \mbox{div} \left( (\nabla u \cdot \nabla \phi) \, \nabla u -\left( \frac{1}{2} \, |\nabla u|^2 - \frac{1}{p+1} u^{p+1} \right) \, \nabla \phi \right) &  &  \\[3mm]
	\displaystyle - \nabla^2 \phi (\nabla u , \nabla u) + \left( \frac{1}{2} \, |\nabla u|^2 - \frac{1}{p+1} u^{p+1} \right)  \, \Delta \phi & = &  0 .
\end{array}
\end{equation}
The result follows immediately from the use of (\ref{eq:PI-1}) to eliminate the terms in $u^{p+1}$ in \eqref{hfggghsjbvcvbcbvnnvmmc}.
\end{proof}

The previous result, together with the divergence theorem implies the following identity~:
\begin{lemma}
Assume that we are given a function $\phi$ (at least $\mathcal C^2$) and $u$ a solution of $\Delta u + u^p=0$ both defined on a bounded smooth domain $\Omega$. Further assume that $u=0$ on $\partial \Omega$, then
\[
\begin{array}{rllll}
	\displaystyle \frac{1}{2} \displaystyle \int_{\partial \Omega}  |\nabla u|^2 \, \nabla \phi \cdot \nu \, {\rm d}\sigma_{\partial \Omega}  + \int_{\Omega} \left( \frac{1}{n} \, \Delta \phi \, |\nabla u|^2 - \nabla^2 \phi (\nabla u , \nabla u) \right) \, {\rm dvol}_{g_\circ} &  & \\[3mm]
	\displaystyle + \left( \frac{n-2}{2n} - \frac{1}{p+1}\right) \, \int_\Omega  |\nabla u|^2   \, \Delta \phi \, {\rm dvol}_{g_\circ}   - \frac{1}{p+1}  \int_\Omega u \, \nabla u \cdot \nabla \Delta \phi  \, {\rm dvol}_{g_\circ} & = &  0 ,
\end{array}
\]
where $\nu$ is the unit normal to $\partial  \Omega$ and $dvol_{g_\circ}$ is defined according to \eqref{HUhisdfjbvbdjbv}.
\label{le:PI-1}
\end{lemma}
\begin{proof}
Just observe that the fact that $(\nabla u \cdot \nabla \phi) \, \nabla u \cdot \nu  =  |\nabla u|^2 \, \nabla \phi \cdot \nu$ since $u$ vanishes on $\partial \Omega$.
\end{proof}

We apply the previous analysis to the function
\[
\phi : =  \frac{1}{2} \, \mbox{dist} (\cdot, \Lambda)^2 ,
\]
and to the domain $\Omega = B_\epsilon (\Lambda)$.

We will need the~:
\begin{lemma}\label{yugvkvghsasassasasssssss}
There exists a constant $C >0$ such that the following estimates hold in $B_\epsilon (\Lambda)$
\begin{equation}\nonumber
\left| \Delta \phi -  n \right| \leq C \,  \epsilon \qquad \mbox{and} \qquad  \left| \nabla \Delta \phi \right| \leq C .
\end{equation}
Moreover, for any $\mathcal C^1$ function $v$ defined on $B_\epsilon (\Lambda)$
\[
\left| \frac{1}{n} \, \Delta \phi \, |\nabla v|^2 - \nabla^2 \phi (\nabla v , \nabla v) \right| \leq C \, \epsilon \, |\nabla v|^2 .
\]
\end{lemma}
\begin{proof}
Follows at once from the expansion of the metric in Fermi coordinates, namely exploiting Lemma \ref{appendilemma1} and  Lemma \ref{frmicordinates},
and the fact that in Fermi coordinates we have
\[
\phi : =  \frac{1}{2} \, \mbox{dist} (\cdot, \Lambda)^2= \frac{|z|^2}{2} .
\]
The lemma follows.
\end{proof}

Now, Poincar\'e inequality in $B_\epsilon( \Lambda)$ reads
\begin{lemma}\label{yugvkvghsasassasasssssss2}
There exists a constant $C>0$, such that, for all $\epsilon \in (0,1)$ and all $u \in H^1_0(B_\epsilon (\Lambda)$, we have
\begin{equation}\label{fggfhdhhjjfjjjjjjjjjj}
\int_{B_\epsilon (\Lambda)} u^2 {\rm dvol}_{g_\circ} \leq C\, \epsilon^2 \, \int_{B_\epsilon (\Lambda)} | \nabla u|^2 \,  {\rm dvol}_{g_\circ}.
\end{equation}
\end{lemma}
\begin{proof}
This follows at once from the fact that the Poincar\'e inequality in the unit ball of $\mathbb R^n$ reads
\[
\int_{B_1} u^2 {\rm d}z \leq C\, \int_{B_1} | \nabla u|^2 \,  {\rm d}z ,
\]
and a scaling argument implies that
\[
\int_{B_\epsilon} u^2 {\rm d}z \leq C\, \epsilon^2 \, \int_{B_\epsilon} | \nabla u|^2 \,  {\rm d}z .
\]
Now, using the product metric $\bar g$ on $B_\epsilon (\Lambda)$, this implies that
\[
\int_{B_\epsilon (\Lambda)} u^2 \, {\rm dvol}_{\bar g} \leq C \, \epsilon^2 \, |\Lambda| \, \int_{B_\epsilon (\Lambda)} | \nabla u|_{g_z}^2 \, {\rm dvol}_{\bar g} .
\]
Finally, since the Euclidean metric and the product metrics are equivalent (see \eqref{eq:defa}), we conclude that
\[
\int_{B_\epsilon (\Lambda)} u^2 \, {\rm dvol}_{g_\circ} \leq C' \, \epsilon^2 \, |\Lambda| \, \int_{B_\epsilon (\Lambda)} | \nabla u|_{g_\circ}^2 \, {\rm dvol}_{g_\circ} .
\]
This completes the proof of \eqref{fggfhdhhjjfjjjjjjjjjj}.
\end{proof}

\begin{proof}[Proof of Theorem \ref{nonexistence}.] Using Lemma \ref{yugvkvghsasassasasssssss} and Lemma \ref{yugvkvghsasassasasssssss2}  together with Cauchy-Schwarz inequality we get
\[
\left| \frac{1}{p+1} \int_{B_\epsilon ( \Lambda)} u \, \nabla u \cdot \nabla \Delta \phi \, {\rm dvol}_{g_\circ} \right|Ê \leq  C \, \epsilon \,  \int_{B_\epsilon (\Lambda)} |\nabla u|^2 \, {\rm dvol}_{g_\circ}.
\]
Collecting these, together with the result of Lemma~\ref{le:PI-1}, we conclude that there exists a constant $C >0$ such that
\[
\begin{array}{rllll}
	\displaystyle  \left(   \frac{n-2}{2}\,-\, \frac{n}{p+1} + C \, \epsilon \right) \, \int_{B_\epsilon (\Lambda)}  |\nabla u|^2    \, {\rm dvol}_{g_\circ} \leq 0 ,
\end{array}
\]
since $\nabla \phi \cdot \nu = \epsilon$ on $\partial B_\epsilon (\Lambda)$.   This implies that $u\equiv 0$ provided $\epsilon$ is close enough to $0$ and $p> \frac{n+2}{n-2}$.
\end{proof}

\begin{remark}
To avoid technical complications we have chosen to carry out the proofs of this section only in the case of a power type nonlinearity. However
the same technique can be easily extended to more general nonlinearities.
\end{remark}

\end{document}